\documentclass[10pt, reqno]{amsart}

\usepackage{amsmath,amssymb,amsthm,amsfonts,verbatim}
\usepackage{microtype}
\usepackage[all,2cell]{xy}
\usepackage{mathtools}
\usepackage{graphicx}
\usepackage{pinlabel}
\usepackage{hyperref}
\usepackage{mathrsfs}
\usepackage{color}
\usepackage{enumerate}
\usepackage{cite}
\usepackage{tcolorbox}

\usepackage[top=1.3in,bottom=1.9in,left=1.2in,right=1.2in]{geometry}

\theoremstyle{plain}
\newtheorem{thm}{Theorem}[section]

\newtheorem{prop}[thm]{Proposition}

\newtheorem{claim}[thm]{Claim}

\newtheorem{qu}{General Problem}

\theoremstyle{definition}

\theoremstyle{remark}

\newcommand{\nc}{\newcommand}
\nc{\dmo}{\DeclareMathOperator}

\DeclareMathOperator{\Diff}{Diff}

\DeclareMathOperator{\FF}{\mathbb{F}}

\DeclareMathOperator{\Inv}{Inv}

\renewcommand{\epsilon}{\varepsilon}
\renewcommand{\tilde}{\tilde}
\nc{\coloneq}{\mathrel{\mathop:}\mkern-1.2mu=}
\nc{\margin}[1]{\marginpar{\scriptsize #1}}
\nc{\red}[1]{\textcolor{red}{#1}}

\nc{\para}[1]{\medskip\noindent\textbf{#1.}}
\title{Actions of $2$-groups of bounded exponent on manifolds}
\author{Lei Chen} 
\email{chenlei@caltech.edu}
\date{Aug 27,  2019}
\begin{document}
\maketitle
\begin{abstract}
In this paper, we show that an infinite 2-group of bounded exponent cannot act faithfully and smoothly on compact manifolds. 
\end{abstract}
\section{Introduction}
We call a group $G$ a 2-group if the order of every element is a power of $2$. We say that $G$ has \emph{bounded exponent} if there is a uniform bound on the orders of elements in $G$. In this paper, we study the action of $2$-groups of bounded exponent on compact manifolds. The main result of this paper is the following:
\begin{thm}\label{main}
If a 2-group of bounded exponent $G$ acts faithfully and smoothly on a compact manifold $M$, then $G$ is a finite group.  
\end{thm}
We call a group a torsion group if every element has finite order. Let $H$ be a fixed group. The Burnside problem for subgroups of $H$ asks whether $H$ contains a finitely generated infinite torsion subgroup. For example, Burnside \cite{Burnside} proved the following well-known result.
\begin{thm}[Burnside's theorem]
If $G$ is a subgroup of $GL(n)$ such that every element has order at most $r$, then $G$ is a finite group.
\end{thm}
Then Schur \cite{Schur} generalizes his result for finitely generated torsion groups without the assumption on the exponent. However, without the ambient group $H$, there exist many finitely generated infinite torsion groups. This is proved by  Golod--Schafarevich \cite{Golod} \cite{GS} and Adian-Novikov \cite{AN}. There is a vast literature on this subject by works of Olshanskii, Ivanov, Grigorchuk, among others.

In this paper, we work with $H=\Diff(M)$ for a compact manifold $M$. The Burnside problem for homeomorphism groups is asked by Ghys and Farb. In dimension one, such infinite torsion group cannot exist as a consequence of H\"{o}lder's theorem \cite[Theorem 2.2.32]{Navas}. Guelman--Lioussse \cite{GL} proved the non-existence of $G$ for homeomorphisms of a genus $g>1$ surface, Hurtado--Kocsard--Rodr\'iguez-Hertz \cite{HKR} for volume preserving diffeomorphisms of the $2$-sphere with the assumption that $G$ has bounded exponent, Rebelo--Silva \cite{RS} for symplectomorphisms on certain symplectic 4-manifolds and Conejeros \cite{Conejeros} for homeomorphisms of the $2$-sphere when $G$ is a $2$-group of bounded exponent.

The special property of a $2$-group different from other torsion groups is observed by Conejeros \cite{Conejeros}. The key property he uses about an infinite $2$-group $G$ is the existence of an involution $g\in G$ such that the centralizer $C(g)$ is still an infinite group. We make use of this key fact as well. 

In this paper, we generalize his result  to higher dimensional manifolds but we restrict our result on smooth actions. Our future goal if to generalize the same result on topological actions. We do not need $G$ to be finitely generated but only has bounded exponent. Also notice that without bounded exponent assumption, there exists an infinite $2$-group acting on the circle (the group consists of order $2^k$ rotations for all $k$). 

The germ group $G^r(d)$ is the group of equivalence classes of regularity $r$ diffeomorphisms of $\mathbb{R}^d$ fixing the origin, where $f\sim g$ if and only if there exists an open neighborhood $U$ such that $f|_U=g|_U$. The key observation of this paper is the following.
\begin{thm}\label{main2}
The smooth germ group $G^r(d)$ for $r>0$ does not contain an infinite torsion group of bounded exponent. 
\end{thm}
This result is an easy application of Reeb stability and Burnside's theorem which only applies for smooth action. We wonder about the following question.
\begin{qu}\label{ques}
Does the topological germ group $G^0(d)$ contains an infinite torsion group of bounded exponent?
\end{qu}
We believe that Problem \ref{ques} is an important question in the study of germ group. Also by the work of this paper, the resolution of Problem \ref{ques} implies Theorem \ref{main} for topological actions. Conejeros \cite{Conejeros} proved that $G^0(2)$ contains no infinite $2$-group of bounded exponent. 

\para{Acknowledgement}
We thank Sebastian Hurtado and Shmuel Weinberger for their help explaining Theorem \ref{Shm}.

\section{The proof of main theorem}
In this section, we prove Theorem \ref{main}. We first prove Theorem \ref{main2} using Reeb stability and Burnside's theorem. Then we give a property of infinite $2$-groups. In the end, we prove Theorem \ref{main} dividing in two cases, one where the action is free and the other when the action has fixed points. For free action, we use a result of Weinberger \cite{Shmuel}.
\subsection{Proof of Theorem \ref{main2}}
We have the following Reeb Stability Theorem \cite{Reeb}. Recall that $G^r(d)$ denotes the group of equivalence class of regularity $r$ diffeomorphisms of $\mathbb{R}^d$ fixing the origin, where $f\sim g$ if and only if there exists an open neighborhood $U$ such that $f|_U=g|_U$. There is a natural projection $p: G^r(d)\to GL(d)$  recording  the derivative of a diffeomorphism at the orgin for $r>0$. 
\begin{thm}[Reeb Stability]
The kernel of the natural projection $p: G^r(d)\to GL(d)$ is torsion-free.
\end{thm}
Reeb Stability is generalized by Thurston \cite{Thurston} saying that the kernel of $p$ is locally indicable. In this paper, we only use the weaker form that Ker$(p)$ contains no torsions. We now prove that $G^r(d)$ contains no infinite torsion group of bounded exponent.
\begin{proof}[Proof of Theorem \ref{main2}]
If $G^r(d)$ contains an infinite torsion group $G$ of bounded exponent, then either $p(G)$ is infinite or $G\cap \text{Ker}(p)$ is infinite. However, $\text{Ker}(p)$ is torsion free, which implies that $p(G)$ is an infinite torsion group of bounded exponent. This contradicts Burnside's theorem.
\end{proof}

\subsection{A property of infinite $2$-groups}
Let $G$ be an infinite $2$-group. We now state a  property of $2$-groups. For $g\in G$, denote by $C(g)$ the centralizer of $g$. Denote by $\Inv(G)$ the set of order $2$ elements (involution) in $G$.
\begin{prop}\label{property}
There exists an element $g\in \Inv(G)$ such that $C(g)$ is also infinite. 
\end{prop}
\begin{proof}We follow the argument of Conejeros \cite{Conejeros}.

If $\Inv(G)$ is finite, then $C(g)$ is infinite for any $g\in \Inv(G)$. Assume that $\Inv(G)$ is infinite, we have an infinite sequence $\{x_n\}_{n=1}^\infty\subset Inv(G)$. Since $x_1,x_n$ either generate a dihedral group or a cyclic group, there exists an element $v_n\in \langle x_1,x_n\rangle$ commuting with $x_1,x_n$. 

If the set $\{v_n\}_{n=1}^\infty\subset G$ is infinite, then we obtain an element $g=x_1$ satisfying $C(g)$ is finite becuase it contains infinitely many different elements. If  $\{v_n\}_{n=1}^\infty\subset G$ is finite, then let $g\in \{v_n\}_{n=1}^\infty$ such that $g=v_n$ for infinitely many $n$. We know that $g$ commutes $x_n$ if $g=v_n$. Then $C(g)$ contains infinitely many different elements $x_n$ for any $n$ such that $g=v_n$.
\end{proof}

\subsection{The proof of Theorem \ref{main}}
In this section, we prove Theorem \ref{main}. Before that, we include an important ingredient. Denote by $\FF_p$ the finite field of order $p$. 
\begin{thm}[Weinberger's theorem]\label{Shm}
For a compact manifold $M$, there exist a natural number $k$ such that $(\FF_p)^k$ cannot act freely on $M$.
\end{thm}
We will include a proof of this theorem in the appendix. We now start the proof of Theorem \ref{main}.
\begin{proof}[Proof of Theorem \ref{main}]We break the proof into two cases: free or not free.

\para{Case 1: free action}
We prove by an induction on the exponent of $G$ that an infinite $2$-group of bounded exponent cannot act freely on a compact manifold.  Assume that if the exponent of $G$ is less than $2^N$, then $G$ cannot act freely on a compact manifold. The case for $N=1$ is given by Theorem \ref{Shm}. Let $\rho: G\to \Diff(M)$ be a free action.

For a group $H$, let $H^2$ be the group generated by squares of $H$. The following short exact sequence is well-known:
\[
1\to H^2\to H\to H_1(H;\mathbb{Z}/2)\to 1.
\]
Now define $G_1=G$ and inductively $G_n=G_{n-1}^2$. By the above exact sequence, we know that $G_n/G_{n-1}$ is an abelian $2$-group. 

Assume that the highest order of $G$ is $2^N$. Therefore $G_N=\{1\}$, which implies that the group $G_{N-1}$ is an abelian $2$-group. By Theorem \ref{Shm}, since $\rho$ is injective, we know that $G_{N-1}$ is a finite group. Therefore we obtain a new manifold $M/G_{N-1}$, which carries a new free action $G/G_{N-1}$. However, the exponent of $G/G_{N-1}$ is less that $2^N$. The inductive assumption implies that $G/G_{N-1}$ is finite. Therefore $G$ is finite.

\para{Case 2: general case}
We prove this case by an induction on the dimension of $M$. The result is true for $\dim(M)=1$ by H\"{o}lder's theorem \cite[Theorem 2.2.32]{Navas}. Assume that for $M$ such that $\dim(M)<n$, the group $\Diff(M)$ contains no finite $2$-groups of bounded exponent. 

Let $\rho: G\to \Diff(M)$ be a faithful action for $G$ an infinite $2$-group of bounded exponent. For $g$ such that $C(g)$ is infinite, we have the following claim. This claim is a consequence of the inductive assumption.
\begin{claim}\label{fixedpoint}
$\rho(g)$ is fixed point free. 
\end{claim}
\begin{proof}
If $\rho(g)$ has fixed points, then the fixed point set $F$ is a finite union of sub-manifolds of $M$ of dimension less than $n$. Since $C(g)$ acts on $F$, there is a finite index subgroup $G'$ of $C(g)$ that preserves each component of $F$. Therefore $G'$ is still an infinite $2$-group of bounded exponent. Let $\rho':G'\to \Diff(F)$ be the new action. By the inductive assumption, we know that $\rho'$ has finite image. Therefore $\text{Ker}(\rho')$ is an infinite $2$-group of bounded exponent. 

We have the natural projection $L: \text{Ker}(\rho')\to G^\infty(n)$ by considering the action near a point $x\in F$. By Theorem \ref{main2}, the image of $L$ is a finite group. Since an element in the kernel of $L$ fixes an open set near $x$ but the fixed point set of a nontrivial finite action has no interior, we know that the kernel of $L$ is torsion-free. This contradicts the fact that $\text{Ker}(\rho')$ is an infinite $2$-group of bounded exponent. 
\end{proof}

We now inductively obtain the following sequences of groups and elements. Let $G_1=G$ and $g_1\in \Inv(G_1)$ be such that $C(g_1)$ is infinite. By Proposition \ref{property}, such $g_1$ exists.  By Claim \ref{fixedpoint}, the action $\rho(g_1)$ is free. Denote by $\rho_1=\rho$. We obtain a new manifold $M_2=M_1/\rho_1(g_1)$ and a new faithful action 
\[
\rho_2:G_2:=C(g_1)/g_1\to \Diff(M_2).\]
Inductively, let $g_n\in G_n$ be such that $C(g_n)\subset G_n$ is infinite. For the same reason, the action $\rho_{n}(g_n)$ on $M_n$ is free. Define $M_{n+1}=M_{n}/\rho_n(g_n)$ and we obtain a new faithful action
\[
\rho_n:G_{n+1}:=G_n/g_n\to \Diff(M_{n+1}).\]
Denote by $p_n: C(g_n)\to C(g_n)/g_n=G_{n+1}$. Let $\widetilde{g_n}\in G$ be an element satisfying the following: 
\begin{itemize}
\item $\widetilde{g_n}\in C(g_1)$, 
\item $p_1(\widetilde{g_n})\in C(g_2)$, 
\item ...
\item $p_{n-2}...p_{1}(\widetilde{g_n})\in C(g_{n-1})$,
\item $p_{n-1}...p_1(\widetilde{g_n})=g_n \in G_n$.
\end{itemize}
Let $K<G$ be the group generated by $\{\widetilde{g_k}\}_{k>0}$ for all $k>0$. 
\begin{claim}
The group $K$ does not depend on the choices of elements $\{\widetilde{g_k}\}_{k>0}$.
\end{claim}
\begin{proof}
Let $K_n$ be the group generated by $\{\widetilde{g_k}\}_{k=1}^n$. We claim that $K_n$ does not depend on the choices of $\{\widetilde{g_k}\}_{k>0}$. We prove this by an induction on $n$. Firstly $K_1$ the the group generated by $g_1$, which does not depend on the choices of $\{\widetilde{g_k}\}_{k>0}$. Assume that $K_{n-1}$ does not depend on the choices of $\{\widetilde{g_k}\}_{k>0}$. Since $\{p_1(\widetilde{g_n})\}_{k=2}^n\subset G_2$ satisfies the same condition as $\{\widetilde{g_k}\}_{k=1}^{n-1}$ in $G_1$. The inductive assumption shows that the group generated by $\{p_1(\widetilde{g_k})\}_{k=2}^n\subset G_2$ does not depend on the choices of $\{p_1(\widetilde{g_k})\}_{k=2}^n\subset G_2$. Notice that elements $\{\widetilde{g_k}\}_{k=2}^n\subset G_1$ are lifts of $\{p_1(\widetilde{g_k})\}_{k=2}^n\subset G_2=C(g_1)/g_1$. Different choices of $\{\widetilde{g_k}\}_{k=2}^n\subset G_1$ are differed by multiplying powers of $g_1$. Therefore the group generated by $\{\widetilde{g_k}\}_{k=1}^n\subset G_1$ does not depend on the choices of $\{\widetilde{g_k}\}_{k=1}^n\subset G_1$.
\end{proof}
At each stage, the action of $\rho_n(g_n)$ on $M_n$ is free. Therefore the action $\rho(K)$ on $M$ is also free. Now, we find an infinite 2-group $K$ of bounded exponent acting freely on $M$. By Case 1, this is not possible.\end{proof}

\section{Appendix}
We now include a proof of Theorem \ref{Shm}.
\begin{proof}[Proof of Theorem \ref{Shm}]
The following argument is based on \cite[Proposition 1]{Shmuel}. The basic idea is that the cohomology of the quotient space is too big. Throughout the whole computation, we use the cohomology with $\FF_p$ coefficient. 

Assume there is a free action of $(\FF_p)^k$ on $M$ for any $k$. Since $M$ is compact, $H^*(M)$ is finite dimensional. Therefore there is a free action of $(\FF_p)^k$ on $M$ for any $k$ such that the action of $(\FF_p)^k$ on $H^*(M)$ is trivial. This is because there is an upper bound $k$ such that $(\FF_p)^k$ is a subgroup of the automorphism group $\text{Aut}(H^*(M))$.

Let $G=(\FF_p)^k$ be a group acting freely on $M$. We have the following spectral sequence \cite[Theorem 7.9]{Brown}:
\[
E_2^{pq}=H^p(G;H^q(M))\Longrightarrow H^{p+q}(M/G)
\]
Since the action of $G$ on $H^*(M)$ is trivial, we have that 
\[
E_2^{pq}=H^p(G)\otimes H^q(M).
\]
The cohomology ring of $\FF_p$ is 
\[
H^*(\FF_p)\cong \FF_p [x]
\] a polynomial ring of one variable such that the degree of $x$ is $2$. Therefore the cohomology of $G$ is 
\[
H^*(G)\cong \FF_p [x_1,...,x_k]\]
such that deg$(x_i)=2$. We compute the dimension:
\[
d_{k,i}:=\dim(H^{2i}(G))={k+i-1 \choose i-1 }
\]
The dimension $\dim(H^{2i}(G))$ is a degree $i$ polynomial of $k$.

Let $n$ be the top dimension of $H^*(M)$. Let $2i>n$, all the differentials in the spectral sequence $E_2^{pq}$ that targets at $H^{2i}(G)\otimes H^0(M)$ are 
\begin{itemize}
\item $H^{2i-2}(G)\otimes H^1(M)$ (of dimension less than $Nd_{k,i-1}$),
\item $H^{2i-4}(G)\otimes H^3(M)$ (of dimension less than $Nd_{k,i-2}$), 
\item ...
\end{itemize}
The sum of dimensions of all of the above terms are bounded by $N(\sum_{j<i} d_{k,j})$. When $k$ is big enough, $N(\sum_{j<i} d_{k,j})$ is less than $d_{k,i}$ since the degree of the polynomial $d_{k,i}$ is larger than $d_{k,j}$ for $j<i$. Therefore, $H^{2i}(G)\otimes H^0(M)$ cannot be fully killed by differentials. This contradicts the fact that $M/G$ is also an $n$-dimensional manifold.

\end{proof}

\bibliographystyle{alpha}
\bibliography{biblio}

\end{document}